\documentclass{amsart}
\usepackage[utf8]{inputenc}
\usepackage{amsfonts}
\usepackage{amssymb}
\usepackage{amsthm}
\usepackage{amsmath}
\usepackage{enumerate}

\theoremstyle{plain}
\newtheorem{theorem}{Theorem}[section]
\newtheorem{lemma}[theorem]{Lemma}

\newtheorem{corollary}[theorem]{Corollary}

\newtheorem{question}[theorem]{Question}
\newtheorem*{stability}{Theorem~\ref{t:stability}}
\newtheorem*{meas}{Theorem~\ref{t:meas}}
\newtheorem*{Effros}{Theorem~\ref{t:Effros}}
\newtheorem*{openL}{Lemma~\ref{l:L}}
\newtheorem*{character}{Theorem~\ref{t:character}}
\newtheorem*{general}{Theorem~\ref{t:character general}}

\theoremstyle{definition}
\newtheorem{definition}[theorem]{Definition}
\newtheorem{example}[theorem]{Example}
\newtheorem{remark}[theorem]{Remark}

\numberwithin{equation}{section}

\newcommand{\R}{\mathbb{R}}

\newcommand{\N}{\mathbb{N}}

\providecommand{\restriction}{\restriction}

\newcommand{\iF}{\mathcal{F}}
\newcommand{\iK}{\mathcal{K}}

\DeclareMathOperator{\length}{length}
\DeclareMathOperator{\diam}{diam}
\DeclareMathOperator{\trees}{Tr}
\DeclareMathOperator{\illfounded}{IF}
\DeclareMathOperator{\finite}{Fin}
\DeclareMathOperator{\MLdim}{dim_{\it{ML}}}
\DeclareMathOperator{\dist}{dist}
\newcommand{\eps}{\varepsilon}
\begin{document}
\title{Stability and measurability of the modified lower dimension}

\author{Rich\'ard Balka}
\address{Alfr\'ed R\'enyi Institute of Mathematics, Re\'altanoda u.~13--15, H-1053 Budapest, Hungary}
\email{balka.richard@renyi.hu}

\author{M\'arton Elekes}
\address{Alfr\'ed R\'enyi Institute of Mathematics, Re\'altanoda u.~13--15, H-1053 Budapest, Hungary AND E\"otv\"os Lor\'and University, Institute of Mathematics, P\'azm\'any P\'eter s. 1/c, 1117 Budapest, Hungary}
\email{elekes.marton@renyi.hu}
\urladdr{http://www.renyi.hu/$\sim$emarci}

\author{Viktor Kiss}
\address{Alfr\'ed R\'enyi Institute of Mathematics, Re\'altanoda u.~13--15, H-1053 Budapest, Hungary}
\email{kiss.viktor@renyi.hu}

\thanks{The first author was supported by the MTA Premium Postdoctoral Research Program and the National Research, Development and Innovation Office -- NKFIH, grant no.~124749. The second author was supported by the National Research, Development and Innovation Office -- NKFIH, grants no.~124749 and 129211.  The third author was supported by the National Research, Development and Innovation Office -- NKFIH, grants no.~124749, 129211, and~128273.}

\subjclass[2010]{Primary 28A75, 28A20}

\keywords{modified lower dimension, finite stability, measurability, Baire class}

\begin{abstract} The lower dimension $\dim_L$ is the dual concept of the Assouad dimension. As it fails to be monotonic, Fraser and Yu introduced the modified lower dimension $\MLdim$ by making the lower dimension monotonic with the simple formula $\MLdim X=\sup\{\dim_L E: E\subset X\}$. 
	
As our first result we prove that the modified lower dimension is finitely stable in any metric space, answering a question of Fraser and Yu. 

We prove a new, simple characterization for the modified lower dimension. For a metric space $X$ let $\iK(X)$ denote the metric space of the non-empty compact subsets of $X$ endowed with the Hausdorff metric. As an application of our characterization, we show that the map $\MLdim \colon \iK(X)\to [0,\infty]$ is Borel measurable. More precisely, it is of Baire class $2$, but in general not of Baire class $1$. This answers another question of Fraser and Yu.


Finally, we prove that the modified lower dimension is not Borel measurable defined on the closed sets of $\ell^1$ endowed with the Effros Borel structure.  

\end{abstract}
\maketitle

\section{Introduction}

The concept of lower dimension was introduced by Larman~\cite{L} under the name minimal dimension. While the Assouad dimension helps to understand the `thickest' part of a set, the lower dimension identifies its `thinnest' part.  

\begin{definition} 
The \emph{lower dimension} of a metric space $X$ is defined as 
\begin{align*} 
\dim_L X = \sup\Big\{\alpha : \text{there is a } C > 0 \text{ such that for all } 0 < r < R \le \diam X 
\\ \text{ and for all } x  \in X \text{ we have } N_r(B(x, R)) \ge C \left(\frac{R}{r}\right)^\alpha\Bigg\},
\end{align*}
where $\diam$ denotes the diameter, $B(x,r)$ is the closed ball of radius $r$ centered at $x$, and $N_r(E)$ is the minimal number of sets of diameter at most $r$ required to cover $E$. We will adopt the convention $\dim_L \emptyset=0$.
\end{definition}
For more information on the lower dimension and applications the reader can consult Fraser's monograph \cite{F} and \cite{BG,HT,JV,KL,S,XYY}. It is easy to see that $\dim_L [0,1]=1$ and $E=[0,1]\cup\{2\}$ satisfies $\dim_L E=0$, so the lower dimension is neither monotone nor contains any information about the `thicker' part of $E$. Fraser and Yu \cite{FY} introduced the following notion which overcomes these shortcomings.  

\begin{definition}
We define the \emph{modified lower dimension} of a metric space $X$ as 
\begin{equation*} 
\MLdim X=\sup\{\dim_L E: E\subset X\}.
\end{equation*}
\end{definition}

Although the modified lower dimension is a new concept, it has already found several applications. First, the modified lower dimension can be used to obtain the best known lower bound on the Assouad dimension of product sets, see \cite[Proposition~4.6]{FY}. Second, the modified lower dimension yields optimal lower bounds in certain intersection theorems, where one considers the intersection of a general set with the set of so-called badly approximable numbers, see \cite[Theorem~14.2.1]{F} and the paragraph preceding \cite[Corollary~14.2.3]{F}. The common theme in both applications is to obtain stronger results for sets whose modified lower dimension is strictly larger than their lower dimension. A well-known family of such sets is the so-called Bedford-McMullen carpets, which have equal Hausdorff and modified lower dimension, while their lower dimension is strictly smaller in general, see \cite[Section~8.3]{F}. 

\bigskip

The first problem  we are dealing with in the present paper concerns stability.
Fraser and Yu \cite[Proposition~4.3]{FY} proved finite stability for properly separated sets, and Fraser \cite[Lemma~3.4.10]{F} showed it for closed sets.

\begin{theorem}[Fraser--Yu, Fraser] Let $(X,\rho)$ be a metric space and let $E,F\subset X$. We have $\MLdim (E\cup F)=\max\{\MLdim E, \, \MLdim F\}$ if either 
\begin{enumerate}
\item $\dist(E,F)=\inf\{ \rho(x,y): x\in E, y\in F\}>0$, or 
\item $X=\R^d$ and $E,F$ are closed. 
\end{enumerate}	
\end{theorem}

Fraser and Yu \cite[Question~9.3]{FY} asked the following, see also \cite[Question~17.1.1]{F}.

\begin{question}[Fraser--Yu] 
Let $X$ be a metric space and $E,F\subset X$, is it true that 
\begin{equation*}
\MLdim (E\cup F)=\max\{\MLdim E, \, \MLdim F\}?
\end{equation*}
\end{question}

In Section~\ref{s:stability} we answer the above question in the affirmative. 

\begin{stability}
Let $X$ be a metric space and let $E,F\subset X$. Then
\begin{equation*}
\MLdim (E\cup F)=\max\{\MLdim E, \, \MLdim F\}.
\end{equation*}
\end{stability}

The main ingredient of the proof is the next lemma about the lower dimension, which might be interesting in its own right. 

\begin{openL} 
Let $(X,\rho)$ be a metric space and let $U\subset X$ be a non-empty open set. Then there exists an open set $V\subset U$ such that $\dim_L V\geq \dim_L X$. 
\end{openL}

This immediately implies that the supremum in the definition of modified lower dimension can be taken for bounded sets only.

\begin{corollary}\label{c:ML}
For any metric space $X$ we have 
\begin{equation*}
\MLdim X=\sup\{ \dim_L E: E\subset X \text{ is bounded}\}.
\end{equation*}
\end{corollary}

\begin{remark}
Some authors define the lower dimension by replacing $R\leq \diam X$ with $R\leq 1$ in the definition. This new dimension remains the same for bounded sets, but can be larger in general. Lemma~\ref{l:L} and Corollary~\ref{c:ML} hold for this notion with the same proof, so the derived notions of modified lower dimension coincide.
\end{remark}

In Section~\ref{s:character} we prove two simple characterization theorems for the modified lower dimension. 
Besides interesting in their own rights, they will play a crucial role in the proof of Theorem~\ref{t:meas} as well. Recall that a natural  number $l$ is identified with the set $\{0, 1, \dots, l-1\}$, the symbol $l^n$ stands for the set of functions from $n$ into $l$, and $l^{<\omega} = \cup_{n\in \N} l^n$ is the set of finite sequences in $l$. For $s,t\in l^{<\omega}$ let 
$s^\frown t$ denote the concatenation of $s$ and $t$.

\begin{definition}\label{d:k,l regularity} 
Let $(X,\rho)$ be a metric space and let $k,l \ge 2$ be integers. Assume that $y_s\in X$ for all ${s \in l^{<\omega}}$. We say that $\{y_s\}_{s \in l^{< \omega}}$ is a \emph{$(k,l)$-regular set} if 
\begin{enumerate}[(i)]
\item \label{i:kl1} $\rho(y_s,y_t) \leq 2^{-kn-1}$ if $s\in l^n$, $t\in l^{n+1}$ and $t$ extends $s$, 
\item \label{i:kl2} $\rho(y_s, y_t)\geq 2^{-kn+2}$ if $s, t \in l^{n}$ and $s \neq t$.
\end{enumerate}
We say that the $(k, l)$-regular set $\{y_s\}_{s \in l^{< \omega}}$ is a \emph{strongly $(k,l)$-regular set} if in addition to conditions \eqref{i:kl1} and \eqref{i:kl2} it satisfies
\begin{enumerate}[(i)]
\setcounter{enumi}{2}
\item \label{i:kl3} $y_{s^\frown 0} = y_s$ for all $s \in l^{<\omega}$.
\end{enumerate}
\end{definition}

We obtain the following useful characterization of the modified lower dimension. 

\begin{general}
Let $(X,\rho)$ be a metric space and let $\alpha\geq 0$. The following statements are equivalent: 
\begin{enumerate}
\item $\MLdim X>\alpha$; 
\item $X$ contains a strongly $(k,l)$-regular set with $ \frac{\log l}{k \log 2}>\alpha$.
\end{enumerate}
\end{general}

The word `strongly' cannot be removed from the above theorem even in separable metric spaces, see Example~\ref{ex}. However, for complete metric spaces this simpler characterization holds.

\begin{character}
Let $(X,\rho)$ be a complete metric space and let $\alpha\geq 0$. The following statements are equivalent: 
\begin{enumerate}
\item  $\MLdim X>\alpha$; 
\item $X$ contains a $(k,l)$-regular set with $ \frac{\log l}{k \log 2}>\alpha$.
\end{enumerate}
\end{character}

In Section~\ref{s:meas} we consider the measurability of the modified lower dimension. Recall that $\iK(X)$ is the space of non-empty compact sets of $X$ endowed with the Hausdorff metric, see Definition~\ref{d:H}. Fraser and Yu \cite[Question~9.3]{FY} asked the next question, see also \cite[Question~17.1.1]{F}.

\begin{question}[Fraser--Yu] Let $X$ be a compact metric space. Is the mapping $\MLdim \colon \iK(X)\to [0,\infty]$ Borel measurable, and, if so, which Baire classes does it belong to?
\end{question} 

We answer the above question based on Theorem~\ref{t:character}. 

\begin{meas} Let $X$ be a metric space. Then $\MLdim \colon \iK(X)\to [0,\infty]$ is Borel measurable. More precisely, it is of Baire class $2$.
\end{meas} 

\begin{remark}
As both finite sets and finite unions of closed balls are dense in $\iK(\R^d)$, the map $\MLdim \colon \iK(\R^d)\to [0, d]$ attains both the values $0$ and $d$ on dense sets, therefore it has no point of continuity. However, it is well known that 
$\iK(\R^d)$ is a Polish space \cite[Theorem~4.25]{Ke}, and real-valued Baire class $1$ functions on Polish spaces have points of continuity \cite[Theorem~24.15]{Ke}. Hence $\MLdim$ is not of Baire class $1$, so Theorem~\ref{t:meas} is sharp in general. 
\end{remark}

In Section~\ref{s:Effros} we consider $\ell^1 = \left\{x \in \R^\omega: \sum_{n=0}^{\infty} |x(n)| < \infty\right\}$ equipped with the norm 
\begin{equation*} \|x\|= \sum_{n=0}^{\infty}  |x(n)|.
\end{equation*} 
Let $\mathcal{F}(\ell^1)$ be the set of closed subsets of $\ell^1$. We endow $\mathcal{F}(\ell^1)$ with the $\sigma$-algebra $\mathcal{B}$ generated by the sets
\begin{equation} \label{e:Effros}
\{F \in \mathcal{F}(\ell^1) : F \cap U \neq \emptyset\},
\end{equation}
where $U$ runs over the open subsets of $\ell^1$. By \cite[Theorem~12.6]{Ke} the measurable space $(\mathcal{F}(\ell^1), \mathcal{B})$ is \emph{standard Borel}, that is, there is a Polish topology on $\mathcal{F}(\ell^1)$ such that its family of Borel sets coincides with $\mathcal{B}$. We prove the following theorem.
\begin{Effros}  
The map $\MLdim \colon \mathcal{F}(\ell^1) \to [0,\infty]$ is not Borel measurable. 
\end{Effros}

\section{Finite stability of the modified lower dimension} \label{s:stability}

The goal of this section is to prove Theorem~\ref{t:stability}. First we need the following lemma. 

\begin{lemma} \label{l:L}
Let $(Y,\rho)$ be a metric space and let $U\subset Y$ be a non-empty open set. Then there exists an open set $V\subset U$ such that $\dim_L V\geq \dim_L Y$. In particular, $\MLdim U\geq \dim_L Y$.
\end{lemma}

\begin{proof} For $y\in Y$ and $r>0$ let $U(y,r)$ be the open ball of radius $r$ centered at $y$, and for $A\subset Y$ let $U(A,r)=\bigcup \{ U(y,r): y\in A\}$.  
Fix $y_0\in U$ and $0<\eps<1/2$ such that $\dist(Y\setminus U,\{y_0\})>2\eps$. Set $V_0=\emptyset$ and $V_1=U(y_0,\eps)$. If $V_n$ is already defined for some $n\geq 1$ then let $V_{n+1}=U(V_n,\eps^{n+1})$. Define
\begin{equation*}
V=\bigcup_{n=1}^{\infty} V_n.
\end{equation*}
Clearly, $V$ is open, and $\sum_{n=1}^{\infty } \eps^n<2\eps$ yields $V\subset U(y_0 ,2\eps)\subset U$.
Therefore, it is enough to prove that  $\dim_L V\geq \dim_L Y$. Let $0<R<3\eps$ and $y\in V$ be arbitrarily given, it is enough to show that there exists $z\in V$ such that 
\begin{equation} \label{e:Bz} U(z,(\eps/3) R)\subset V\cap U(y,R).
\end{equation} 
Let $k,m$ be the unique positive integers such that $y\in V_k\setminus V_{k-1}$ and \begin{equation} \label{e:3} 
3\eps^{m+1}\leq R<3\eps ^m.
\end{equation}
If $m\geq k$ then $U(y,(\eps/3)R)\subset U(y,\eps^{k+1})\subset V_{k+1}\subset V$, so $z=y$ satisfies \eqref{e:Bz}.

Finally, assume that $k>m$. Define $y_1,\dots,y_k\in V$ such that $y_k=y$ and $y_{n}\in  V_n\cap U(y_{n-1},\eps^{n})$ for all $1\leq n\leq k$. Let $z=y_m$. Then \eqref{e:3} and $y_m\in V_m$ imply that
\begin{equation} \label{e:cont1}
U(z,(\eps/3)R)\subset U(y_{m},\eps^{m+1})\subset V_{m+1} \subset V.
\end{equation} 
Inequalities
\begin{equation*} 
\rho(y,y_m)<\sum_{n=m+1}^{k} \eps^n<2\eps^{m+1}
\end{equation*}  
and \eqref{e:3} yield 
\begin{equation} \label{e:cont2}
U(z,(\eps/3)R) \subset U(y_m,\eps^{m+1}) \subset U(y,3\eps^{m+1})  \subset U(y,R).
\end{equation} 
Then \eqref{e:cont1} and \eqref{e:cont2} imply \eqref{e:Bz}, and the proof is complete.
\end{proof}

\begin{theorem} \label{t:stability}
Let $X$ be a metric space and let $E,F\subset X$. Then
\begin{equation*}
\MLdim (E\cup F)=\max\{\MLdim E, \, \MLdim F\}.
\end{equation*}
\end{theorem}

\begin{proof} Clearly, $\MLdim (E\cup F)\geq \max\{\MLdim E, \, \MLdim F\}$ holds by the monotonicity of modified lower dimension, so it is enough to prove that 
\begin{equation} \label{e:mon}
\MLdim (E\cup F)\leq \max\{\MLdim E, \, \MLdim F\}.
\end{equation} 
Assume to the contrary that \eqref{e:mon} does not hold. By shrinking $E$ and $F$ if necessary we may suppose that
\begin{equation} \label{e:<} \max\{\MLdim E, \, \MLdim F\}<\dim_L (E\cup F).
\end{equation} 
First assume that $E$ is dense in $E\cup F$. The lower dimension is stable under taking closures, see \cite[Theorem~2.3]{F2}. Therefore, we obtain
\begin{equation*}  \MLdim E\geq \dim_L E=\dim_L \overline {E}=\dim_{L} \overline{E\cup F}=\dim_L (E\cup F),
\end{equation*}
which contradicts \eqref{e:<}. Finally, suppose that $E$ is not dense in $E\cup F$. Then there is an open set $U\subset X$ such that $U\cap F=U\cap (E\cup F)\neq \emptyset$. Applying Lemma~\ref{l:L} for $Y=E\cup F$ implies that $\MLdim (U\cap (E\cup F)\geq \dim_L (E\cup F)$. By the monotonicity of modified lower dimension we obtain 
\begin{equation*} \MLdim F \geq \MLdim (U\cap F)=\MLdim (U\cap (E\cup F))\geq \dim_L (E\cup F),
\end{equation*}
which contradicts \eqref{e:<}. The proof is complete.
\end{proof}

\section{Characterizations of the modified lower dimension} \label{s:character}

The goal of this section is to prove Theorems~\ref{t:character general} and 
\ref{t:character}. 

\begin{theorem}\label{t:character general} 
Let $(X,\rho)$ be a metric space and let $\alpha\geq 0$. The following statements are equivalent: 
\begin{enumerate}
\item \label{i.s:01} $\MLdim X>\alpha$; 
\item \label{i.s:02} $X$ contains a strongly $(k,l)$-regular set with $ \frac{\log l}{k \log 2}>\alpha$.
\end{enumerate}
\end{theorem}
\begin{proof} First we prove $\eqref{i.s:01} \Rightarrow \eqref{i.s:02}$. We may assume by shrinking $X$ if necessary that $\dim_L X>\alpha$. Choose $\beta$ so that $\alpha < \beta < \dim_L X$. The definition of lower dimension implies that there exists a constant $C > 0$ such that for all $0<r<R \le \diam X$ and $x \in X$ we have
$N_r(B(x, R)) \geq C\left(\frac{R}{r}\right)^{\beta}$.
If $\diam X < 1$, then by replacing $C$ with $C (\diam X)^{\beta}$ we obtain that the above property holds for all $R \le 1$, that is, 
\begin{equation}\label{e:lower dim for R <= 1}
N_r(B(x, R)) \ge C\left(\frac{R}{r}\right)^{\beta} \text{ for all }  0 < r < R \le 1 \text{ for all } x \in X.
\end{equation}
As $\beta>\alpha$, we can fix a large enough integer $k$ such that $k\geq 5$ and
\begin{equation*} l=\left\lfloor C\cdot 2^{(k-4)\beta} \right\rfloor \text{ satisfy } \frac{\log l}{k \log 2} > \alpha,
\end{equation*}
where $\lfloor \cdot \rfloor$ denotes the integer part. Let us pick any $y_{\emptyset} \in X$ and suppose that $y_s$ is already defined for some $n\in \N$ and $s \in l^{n}$. Let $c\in \{0,\dots, l-1\}$, we need to define $y_{s^ \frown c}\in X$ satisfying the conditions of Definition~\ref{d:k,l regularity}. This will give the required $(k,l)$-regular set. 
	
Let $R = 2^{-kn-1}$ and $r =2^{-k(n+1)+3}$. By $k \ge 5$ we obtain $0 < r < R \le 1$. Let $Z$ be a maximal family of points from $B(y_s, R)$ containing $y_s$ such that 
\begin{equation} \label{e:rhoz} \rho(z, z') \geq  2^{-k(n+1)+2} \text{ for distinct points } z, z' \in Z.
\end{equation}  
It is enough to prove that $\#Z\geq l$. Indeed, then we can choose distinct points $y_{s^\frown c}\in Z$ for all $c\in \{0,\dots,l-1\}$ with $y_{s^\frown 0} = y_s$. By $Z\subset B(y_s,R)$ we obtain $\rho(y_s,y_{s^\frown c})\leq 2^{-kn-1}$ for all $c$, so Definition~\ref{d:k,l regularity}~\eqref{i:kl1} holds. Inequality~\eqref{e:rhoz} yields that the points $y_{s^\frown c}$ satisfy Definition~\ref{d:k,l regularity}~\eqref{i:kl2}, while Definition \ref{d:k,l regularity} \eqref{i:kl3} clearly holds for $s$.
	
Finally, we prove $\#Z\geq l$. Using the maximality of $Z$ we obtain that
\begin{equation}\label{e:covering of K with big balls}
B(y_s, R)  \subset \bigcup_{z \in Z} B\left(z,2^{-k(n+1)+2}\right).
\end{equation}
As the union in \eqref{e:covering of K with big balls} is a covering of $B(y_s, R)$ with sets of diameter at most $r$, using also \eqref{e:lower dim for R <= 1} we obtain that 
\begin{align*}
\#Z \ge N_r(B(y_s, R)) \ge C \left(\frac{R}{r}\right)^{\beta} =C\cdot 2^{(k-4)\beta}\geq l, \end{align*}
and the proof of $\eqref{i.s:01} \Rightarrow \eqref{i.s:02}$ is complete. 

Now we prove $\eqref{i:02} \Rightarrow \eqref{i:01}$. Assume that $Y = \{y_s\}_{s \in l^{<\omega}}$ is a strongly  $(k,l)$-regular set in $X$ with $\frac{\log l}{k \log 2}>\alpha$. 
	By the definition of modified lower dimension it is enough to show for \eqref{i.s:01} that
	\begin{equation} \label{e.s:LE} 
	\dim_L Y \ge \frac{\log l}{k \log 2}.
	\end{equation} 
	We will check the definition for arbitrarily given $0<r < R \le \diam Y\leq 2$ and $y_u \in Y$ for some $u\in l^{<\omega}$. Let $n\geq 1$ and $m\geq -1$ be the unique integers  such that  
	\begin{equation} \label{e.s:rR}
	2^{-kn+1} < R \le 2^{-k(n-1)+1} \quad \text{and} \quad 
	2^{-k(n+m+1)+1} \leq r < 2^{-k(n+m)+1}.
	\end{equation} 
	Let $s\in l^n$ with $s = u \restriction n$ if $\length(u) \ge n$ and $s = u^\frown (0, \dots, 0)$ if $\length(u) < n$, where $\length(u)$ denotes the number of coordinates of $u$ and $u \restriction n$ is the restriction of $u$ to its first $n$ coordinates. Note that by Definition \ref{d:k,l regularity}~\eqref{i:kl1} if $t$ extends $s$ then 
	\begin{equation}
	    \label{e:t extends s}
	    \rho(y_s, y_t) < 2^{-kn}.
	\end{equation}
	Note also that if $\length(u) < n$ then $y_s = y_u$, so in either case, 
	\begin{equation}
	    \label{e:dist s u}
	    \rho(y_s, y_u) < 2^{-kn}.
	\end{equation}
	By Definition~\ref{d:k,l regularity}~\eqref{i:kl2} we obtain 
\begin{equation*} \rho(y_t, y_{t'}) \geq 2^{-k(n+m)+2} \text{ for all distinct sequences } t, t' \in l^{n+m},
\end{equation*} 
so each set of diameter $r<2^{-k(n+m)+1}$ can contain at most one of these points. We claim that 
	\begin{equation} \label{e.s:Nr}  
	N_r(B(y_u, R)\cap Y) \ge l^m.
	\end{equation}  
	The inequality is clear if $m = -1$. Let $m \ge 0$ be fixed. The number of sequences $t \in l^{n+m}$ that extend $s$ is exactly $l^m$, and each such $y_t$ satisfies $\rho(y_s, y_t) < 2^{-kn}$ by \eqref{e:t extends s}. Thus \eqref{e:dist s u} yields that $y_t$ is contained in the ball $B(y_u, R)$, so \eqref{e.s:Nr} holds. 
	
	Hence \eqref{e.s:rR} and \eqref{e.s:Nr} yield that $\alpha = \frac{\log l}{k \log2}$ satisfies 
	\begin{equation*} 
	\left(\frac{R}{r}\right)^\alpha \le 2^{\alpha k(m+2)} = l^{2+m} \leq l^2 N_r(B(y_u, R)\cap Y).
	\end{equation*}
	Hence applying the definition of the lower dimension with the absolute constant $C = l^{-2}$ yields \eqref{e.s:LE}. This completes the proof of the theorem.
\end{proof}

\begin{example} \label{ex}
The word `strongly' cannot be removed from Theorem~\ref{t:character general}, even in separable metric spaces: For all $n\in \N$ and $s\in 2^n$ let us define 
\begin{equation*} 
y_s=\sum_{i=0}^{n-1} (2s(i)-1)2^{-2i-1},
\end{equation*} 
where we assume that $y_\emptyset = 0$.
Let $X=\{y_s\}_{s\in 2^{<\omega}}$ be a countable metric space with the inherited metric from $\R$. It is easy to see that $X$ is a $(2,2)$-regular set in itself. However, it is also not hard to see that $X$ is discrete, so every subset of $X$ is of lower dimension $0$, hence $\MLdim X=0$.
\end{example}

However, if we assume completeness then we can replace strong regularity by regularity. 

\begin{theorem}\label{t:character} 
	Let $(X,\rho)$ be a complete metric space and let $\alpha\geq 0$. The following statements are equivalent: 
	\begin{enumerate}
		\item \label{i:01} $\MLdim X>\alpha$; 
		\item \label{i:02} $X$ contains a $(k,l)$-regular set with $ \frac{\log l}{k \log 2}>\alpha$.
	\end{enumerate}
\end{theorem}
\begin{proof} Implication \eqref{i:01} $\Rightarrow$ \eqref{i:02} follows from the analogous implication of Theorem \ref{t:character general}. Hence, it remains to prove $\eqref{i:02} \Rightarrow \eqref{i:01}$. Assume that $\{y_s\}_{s \in l^{<\omega}}$ is a $(k,l)$-regular set in $X$ with $\frac{\log l}{k \log 2}>\alpha$. For all $n\in \N$ and $s\in l^n$ let $D_s=B(y_s,2^{-kn})$. Let us define
\begin{equation*} 
K=\bigcap_{n =0}^{\infty} \bigcup_{s \in l^n} D_s.
\end{equation*} 
By the definition of modified lower dimension it is enough to show for \eqref{i:01} that
\begin{equation} \label{e:LE} 
\dim_L K \ge \frac{\log l}{k \log 2}.
\end{equation} 
We will check the definition for arbitrarily given $0<r < R \le \diam K\leq 2$ and $x \in K$. Let $n\geq 1$ and $m\geq -1$ be the unique integers  such that  
\begin{equation} \label{e:rR}
2^{-kn+1} < R \le 2^{-k(n-1)+1} \quad \text{and} \quad 
2^{-k(n+m+1)+1} \leq r < 2^{-k(n+m)+1}.
\end{equation} 
Choose the unique $s\in l^n$ with $x\in D_{s}$. Then $R> 2^{-kn+1}$ yields $D_{s} \subset B(x, R)$. By Definition~\ref{d:k,l regularity}~\eqref{i:kl2} we obtain 
$\dist(D_t, D_{t'}) \geq 2^{-k(n+m)+1}$ for all distinct sequences $t, t' \in l^{n+m}$, so each set of diameter $r<2^{-k(n+m)+1}$ can intersect at most one of these balls. As $D_t\cap K\neq \emptyset$ and the number of sets $D_t$ with $t \in l^{n+m}$ and $D_t \subset D_s$ is $l^m$, we obtain 
\begin{equation} \label{e:Nr}  
N_r(B(x, R)\cap K) \ge l^m,
\end{equation}  
which holds for $m=-1$ as well. 
Thus \eqref{e:rR} and \eqref{e:Nr} yield that $\alpha = \frac{\log l}{k \log2}$ satisfies 
\begin{equation*} 
\left(\frac{R}{r}\right)^\alpha \le 2^{\alpha k(m+2)} = l^{2+m} \leq l^2 N_r(B(x, R)\cap K).
\end{equation*}
Hence applying the definition of the lower dimension with the absolute constant $C = l^{-2}$ yields \eqref{e:LE}, and the proof of the theorem is complete. 
\end{proof}

\section{Measurability of the modified lower dimension} \label{s:meas} 
The goal of this section is to prove Theorem~\ref{t:meas} using a  characterization from the previous section. First we need some preparation. 

\begin{definition} 
Let $Y$ be a metric space and let $A\subset Y$. We say that $A$ is $F_{\sigma}$ if it is a countable union of closed sets, $A$ is $G_{\delta}$ if it is a countable intersection of open sets, and $A$ is $G_{\delta \sigma}$ if it is a countable union of $G_{\delta}$ sets. We say that $f\colon Y\to [0,\infty]$ is of \emph{Baire class $2$} if the sets $\{y\in Y: f(y)>\alpha\}$ and $\{y\in Y: f(y)<\alpha\}$ are $G_{\delta \sigma}$ for all $\alpha\geq 0$.
\end{definition}

\begin{definition} \label{d:H}
For a metric space $(X,\rho)$ let $(\mathcal{K}(X),d_{H})$ be the set of non-empty compact subsets of $X$ endowed with the \emph{Hausdorff metric}, that is, for every $K_1,K_2\in \mathcal{K}(X)$ we have
\begin{equation*} 
d_{H}(K_1,K_2)=\min \left\{r: K_1\subset B(K_2,r) \textrm{ and } K_2\subset B(K_1,r)\right\},
\end{equation*}
where $B(A,r)=\{x\in X: \exists y\in A \textrm{ such that } \rho(x,y)\leq r\}$.
\end{definition}

\begin{theorem} \label{t:meas} Let $X$ be a metric space. Then $\MLdim \colon \iK(X)\to [0,\infty]$ is Borel measurable. More precisely, it is of Baire class $2$.
\end{theorem} 
\begin{proof}

Let $(X,\rho)$ be a metric space, we prove that $\MLdim \colon \iK(X)\to [0,\infty]$ is of Baire class $2$. It is enough to show that $\{K \in \mathcal{K}(X) : \MLdim K > \alpha\}$ is an $F_\sigma$ set for each $\alpha \geq 0$. Indeed, this would imply that $\{K \in \mathcal{K}(X) : \MLdim K < \alpha\}$ is a $G_{\delta \sigma}$ set for all $\alpha\geq 0$, so $\MLdim$ is of Baire class 2. For integers $k,l\geq 2$ let
\begin{equation*}
\iF(k,l)=\{K \in \mathcal{K}(X) : K \text{ contains a  $(k,l)$-regular set}\}.
\end{equation*}  
Fix $\alpha\geq 0$. Theorem~\ref{t:character} implies that  
\begin{equation*}
\{K \in \mathcal{K}(X) : \MLdim K > \alpha\} = \bigcup \left\{ \iF(k,l): k,l\geq 2,~\frac{\log l}{k \log 2} > \alpha \right\}.
\end{equation*}
Hence it is enough to show that the sets $\iF(k,l)$ are closed. Fix $k$ and $l$ and assume that $K_i \in \iF(k,l)$ for each $i$ and $K_i \to K$ in the Hausdorff metric, it remains to show that $K \in  \iF(k,l)$. For all $i\in \N$ let $\{y^i_s\}_{s \in l^{<\omega}}$ be a $(k,l)$-regular set in $K_i$. By successively taking subsequences and considering the diagonal sequence, we may suppose that $y^i_s\to y_s$ as $i\to \infty$ for each $s \in l^{<\omega}$ with some $y_s \in K$. Then it is straightforward that $\{y_s\}_{s \in l^{<\omega}}$ is a $(k,l)$-regular set in $K$, so $K\in \iF(k,l)$. The proof of the theorem is complete. 
\end{proof}

\section{A non-measurability result in $\ell^1$} \label{s:Effros}

The goal of this section is to prove the following theorem.

\begin{theorem} \label{t:Effros} 
The map $\MLdim \colon \mathcal{F}(\ell^1) \to [0,\infty]$ is not Borel measurable. 
\end{theorem}
\begin{proof}
We denote the space of subtrees of $\N^{<\omega}$ by $\trees$, that forms a Polish space when equipped with the subspace topology coming from $2^{\N^{<\omega}}$, see \cite[4.32]{Ke}. Recall the definition of the Borel $\sigma$-algebra $\mathcal{B}$ on $\mathcal{F}(\ell^1)$ from \eqref{e:Effros}. We prove the theorem by constructing a Borel measurable mapping $\overline{\varphi} \colon \trees \to \mathcal{F}(\ell^1)$ (that is, $\overline{\varphi}^{-1}(B)$ will be a Borel subset of $\trees$ for each $B \in \mathcal{B}$) with the property
\begin{equation}\label{e:T in IF <-> dim_ML phi T > 0}
T \in \illfounded~\Leftrightarrow~\MLdim\overline{\varphi}(T) > 0,
\end{equation}
where $\illfounded = \{T \in \trees : T \text{ has an infinite branch}\}$. Since $\illfounded$ is not a Borel subset of $\trees$ by \cite[27.1]{Ke}, it follows from \eqref{e:T in IF <-> dim_ML phi T > 0} that 
\begin{equation*}  \{F \in \mathcal{F}(\ell^1) : \MLdim F > 0\} \not \in \mathcal{B},
  \end{equation*} thus $\MLdim$ is not Borel measurable. 
  
Thus it is enough to construct a Borel measurable map $\overline{\varphi} \colon \trees \to \mathcal{F}(\ell^1)$ satisfying \eqref{e:T in IF <-> dim_ML phi T > 0}. We do so by first constructing a map $\varphi \colon \N^{<\omega} \to \finite(\ell^1)$ assigning to each element $u \in \N^{<\omega}$ a finite subset $\varphi(u)$ of $\ell^1$. We then define $\overline{\varphi} \colon \trees \to \mathcal{F}(\ell^1)$ by
\begin{equation*} 
\overline{\varphi}(T) = \overline{\bigcup\{\varphi(u) : u \in T\}}.
\end{equation*} 
We first check that the map $\overline{\varphi}$ is necessarily Borel measurable. It is enough to check that for each open set $U \subset \ell^1$ the inverse image of the set in \eqref{e:Effros} is a Borel subset of $\trees$. Let us fix an open set $U \subset \ell^1$. Since the closure of a set $E$ intersects $U$ if and only if $E$ intersects $U$, we obtain that 
\begin{equation*}
\{T : \overline{\varphi}(T) \cap U \neq \emptyset\} = 
\{T : \exists u \in T, ~ \varphi(u) \cap U \neq \emptyset\} = \bigcup_{ u \in \N^{<\omega},\, \varphi(u) \cap U \neq \emptyset} \{T : u \in T\},
\end{equation*}
 which is clearly an open subset of $\trees$. 
  
It remains to define the map $\varphi \colon \N^{<\omega} \to \finite(\ell^1)$ so that \eqref{e:T in IF <-> dim_ML phi T > 0} is satisfied for the resulting $\overline{\varphi}$. First, for each $u \in \N^{<\omega}$ choose indices $n_u^0, n_u^1 \in \N$  such that $n_u^i \neq n_v^j$ if $(u,i) \neq (v,j)$. Next, we define $\varphi$ recursively. Let $\varphi(\emptyset) =\{\mathbf{0}\}$, where $\mathbf{0}\in \ell^1$ is the zero vector. If $\varphi(u)$ is defined for some $u \in \N^n$ and $v \in \N^{n+1}$ extends $u$, then let
 \begin{equation*} 
 \varphi(v) = \left\{x + 2^{-2n-1}\chi(n_v^i): x \in \varphi(u),\, i\in \{0,1\}\right\},
 \end{equation*} 
 where $\chi(n)$ is the unit vector whose only non-zero coordinate has index $n$. 

As the indices $n^i_u$ are distinct, if the number of coordinates of $u$ and $v$ satisfy $\length(u), \length(v)>k$ and $u(k) \neq v(k)$, then 
\begin{equation}\label{e:distance for phi}
 \|x-y\| \ge 2^{-2k} \text{ for all } x \in \varphi(u) \text{ and
  } y\in \varphi(v).
 \end{equation}
First, we claim that if $T$ does not have an infinite branch, that is, $T\not\in \illfounded$, then $\bigcup_{u \in T} \varphi(u)$ does not have limit points. Let $\{x_n\}_{n \in \omega}$ such that $x_n \in \varphi(u_n)$ for some $u_n \in T$ be any sequence of points from $\bigcup_{u \in T} \varphi(u)$. As $T$ does not have an infinite branch, there exists $k$ such that for each $N$ there exist $n, m \ge N$ such that $\length(u_n), \length(u_m) > k$ and $u_n(k) \neq u_m(k)$. Then Inequality~\eqref{e:distance for phi} yields that $\|x_n - x_m\| \ge 2^{-2k}$, hence $\{x_n\}_{n \in \omega}$ is not a Cauchy-sequence, showing that $\bigcup_{u \in T} \varphi(u)$ does not have limit points. Therefore its closure, $\overline{\varphi}(T)$ is countable. Any non-empty subset of a countable closed set has isolated points, therefore $\MLdim \overline{\varphi}(T) = 0$ in this case. 
  
Now suppose that $T \in \illfounded$, and $z \in \N^\omega$ is an infinite branch of $T$. We show that the closed set 
\begin{equation*} 
E=\overline{\{\varphi(z \restriction n) : n \in \omega\}}
\end{equation*}contains a $(2,2)$-regular family $\{y_s\}_{s\in 2^{<\omega}}$. Then applying Theorem~\ref{t:character} for the complete metric space $E$ implies that
\begin{equation*} 
\MLdim \overline{\varphi}(T)\geq \MLdim E>0.
\end{equation*} 
Let $y_{\emptyset}=\mathbf{0}$ and assume by induction that $y_s \in \varphi(z\restriction n)$ is defined for all $s \in 2^n$ for some $n\ge 0$. For $c\in \{0,1\}$ define $y_{s^\frown c}\in \varphi(z\restriction (n+1))$ as 
\begin{equation} \label{e:ys} y_{s^ \frown c} = y_s + 2^{-2n-1} \chi(n_{z \restriction (n+1)}^c).
\end{equation} 
It remains to check that $\{y_s\}_{s \in 2^{<\omega}}$ is indeed a $(2, 2)$-regular set. By \eqref{e:ys} we obtain $\|y_s-y_{s^\frown c}\| = 2^{-2n-1}$, so Definition~\ref{d:k,l regularity}~\eqref{i:kl1} holds.
If $s, t \in 2^n$ with $s \neq t$ then $s(k) \neq t(k)$ for some $k \le n-1$, hence \eqref{e:distance for phi} implies that $\|y_s-y_t\| \ge 2^{-2k}\geq 2^{-2n+2}$, so Definition~\ref{d:k,l regularity}~\eqref{i:kl2} is satisfied, too. This completes the proof of the theorem. 
\end{proof}

\subsection*{Acknowledgments}
We are indebted to Jonathan M.~Fraser for some illuminating conversations.

\end{document}